\documentclass[12pt,reqno]{amsart}

\usepackage{t1enc}
\usepackage{amssymb,units,url}
\usepackage{booktabs}
\usepackage[bookmarks=false,unicode,colorlinks,urlcolor=red,citecolor=red,linkcolor=blue]{hyperref}
\usepackage{aliascnt}
\usepackage{tikz}
\usepackage{pgfplots} 

\usepackage{color}

\usepackage[margin=1.25in]{geometry}

\newtheorem{theorem}{Theorem}

\newaliascnt{lemma}{theorem}

\aliascntresetthe{lemma}

\newaliascnt{proposition}{theorem}
\newtheorem{proposition}[proposition]{Proposition}
\aliascntresetthe{proposition}

\newaliascnt{corollary}{theorem}

\aliascntresetthe{corollary}

\newaliascnt{conjecture}{theorem}

\aliascntresetthe{conjecture}

\newaliascnt{example}{theorem}

\aliascntresetthe{example}

\makeatletter
\def\tagform@#1{\maketag@@@{\ignorespaces#1\unskip\@@italiccorr}}
\let\orgtheequation\theequation
\def\theequation{(\orgtheequation)}
\makeatother
\def\equationautorefname~{}

\newcommand{\arxiv}[1]{%
 \href{http://front.math.ucdavis.edu/#1}{ArXiv:#1}}
\newcommand{\mref}[1]{%
\href{http://www.ams.org/mathscinet-getitem?mr=#1}{#1}}

\newcommand{\R}{{\mathbb R}}
\newcommand{\Rd}{{\mathbb R}^d}

\begin{document}

\title{P\'{o}lya's conjecture fails for the fractional Laplacian}
\author[]{Mateusz Kwa\'{s}nicki, Richard S. Laugesen and Bart{\l}omiej A. Siudeja}
\address{Faculty of Pure and Applied Mathematics, Wroc{\l}aw University of Science and Technology, ul.\@ Wybrze\.ze Wyspia\'nskiego 27, 50-370 Wroc{\l}aw, Poland}
\address{Department of Mathematics, University of Illinois, Urbana,
IL 61801, U.S.A.}
\address{142 E 32nd Ave, Eugene, OR 97405, U.S.A.}
\email{mateusz.kwasnicki\@@pwr.edu.pl  Laugesen\@@illinois.edu  siudej\@@gmail.com}
\date{\today}


\begin{abstract}
The analogue of P\'{o}lya's conjecture is shown to \emph{fail} for the fractional Laplacian $(-\Delta)^{\alpha/2}$ on an interval in $1$-dimension, whenever $0<\alpha<2$. The failure is total: every eigenvalue lies below the corresponding term of the Weyl asymptotic.

In $2$-dimensions, the fractional P\'{o}lya conjecture fails already for the first eigenvalue, when $0 < \alpha < 0.984$.
\end{abstract}

\maketitle

\subsection*{\bf Introduction}
The Weyl asymptotic for the $n$-th eigenvalue of the Dirichlet Laplacian on a bounded domain of volume $V$ in $\R^d$ says that
\[
\lambda_n \sim (nC_d/V)^{2/d} \qquad \text{as $n \to \infty$,}
\]
where $C_d = (2\pi)^d/\omega_d$ 
and $\omega_d=$ volume of the unit ball in $\Rd$. In $1$-dimension, ``volume'' means length and in $2$-dimensions it means area, so that $C_1 = \pi, C_2 =  4\pi$. P\'{o}lya suggested that the Weyl asymptotic provides more than a limiting relation. He conjectured that it gives a lower bound on each eigenvalue:
\[
\lambda_n \geq (nC_d/V)^{2/d} , \qquad n=1,2,3,\ldots .
\]
He proved this inequality for tiling domains \cite{Pol61}, but it remains open in general.  

In this note, we deduce from existing results in the literature that the analogue of 
P\'{o}lya's conjecture \emph{fails} for the fractional Laplacian $(-\Delta)^{\alpha/2}$ on the simplest domain imaginable --- an interval in $1$-dimension. In $2$-dimensions we show it fails on the disk and square, at least for some values of $\alpha$. 

\subsection*{\bf Fractional P\'{o}lya conjecture}
The fractional Laplacian $(-\Delta)^{\alpha/2}$ is a Fourier multiplier operator, with
\[
\big( (-\Delta)^{\alpha/2} u \big)\widehat{\ }(\xi) = |\xi|^\alpha \widehat{u}(\xi) , \qquad \alpha > 0 ,
\]
where the Fourier transform is defined by
\[
\widehat{u}(\xi) = \frac{1}{(2\pi)^{d/2}} \int_{\Rd} u(x) e^{-i x \cdot \xi} \, dx .
\]
The fractional Laplacian is known to have discrete Dirichlet spectrum on the bounded domain $\Omega \subset \R^d$, with weak eigenfunctions belonging to the fractional Sobolev space
\[
H^{\alpha/2}_0(\Omega) = \{ u \in H^{\alpha/2}(\Rd) : \text{$u=0$ a.e.\ on $\Rd \setminus \Omega$} \, \} .
\]
For further information on the fractional Sobolev space see \cite{DNPV12}; for the fractional Laplacian see \cite{Kwa}; and for the variational formulation of the spectrum see \cite{Fra}. 

Write $\lambda_n(\alpha)$ for the $n$-th eigenvalue of $(-\Delta)^{\alpha/2}$ on $\Omega$. The Weyl asymptotic (see \cite[Theorem~3.1]{Fra} and associated references) says that
\begin{equation} \label{eq:Weyl}
\lambda_n(\alpha) \sim (nC_d/V)^{\alpha/d} \qquad \text{as $n \to \infty$.}
\end{equation}
Thus the fractional analogue of the P\'{o}lya conjecture is the assertion that
\[
\lambda_n(\alpha) \geq (nC_d/V)^{\alpha/d} , \qquad n=1,2,3,\ldots .
\]
This inequality is what we shall disprove. 

\subsection*{\bf Fractional P\'{o}lya conjecture fails for the unit interval, for all eigenvalues}
In $1$-dimension on an interval of length $L$, the conjecture says $\lambda_n(\alpha) \geq (n\pi/L)^\alpha$. Equality holds when $\alpha=2$, the classical case of a vibrating string, but the equality is broken as soon as $\alpha$ drops below $2$, according to the next theorem.    
\begin{theorem}[Interval] \label{th:polyafalse}
Suppose $\Omega=(0,L)$ is an interval in $1$-dimension, and let $0<\alpha<2$. Then $\lambda_n(\alpha) < (n\pi/L)^\alpha$ for all $n$. 
\end{theorem}
Hence the fractional P\'{o}lya conjecture fails on intervals, which contradicts a claim made about tiling domains (in all dimensions) in the literature \cite{YYY12}. See also our remark later in the paper about the square, which is a tiling domain in $2$ dimensions. 
\begin{proof}
The eigenvalues of the fractional Laplacian are known to be bounded above by powers of the usual Laplacian eigenvalues, with strict inequality: 
\[
\lambda_n(\alpha) < \lambda_n(2)^{\alpha/2} , \qquad n=1,2,3,\ldots ,
\]
whenever $0<\alpha<2$. See \autoref{pr:spectralcomparison} and the discussion at the end of the paper. 

On an interval in $1$-dimension this last inequality says $\lambda_n(\alpha) < (n\pi/L)^\alpha$, which proves the theorem.  
 
For an alternative proof when $\alpha=1$ that provides more explicit estimates, we recall an estimate of Kulczycki, Kwa\'{s}nicki, Ma{\l}ecki and Stos \cite[Theorem~6]{KKMS10}. It implies for the interval of length $L=2$ that
\[
\lambda_n(1) < \frac{n\pi}{L} - \frac{\pi}{40} 
\]
whenever $n \geq 4$. When $n=1,2,3$, those authors give the following numerical estimates \cite[Section 11]{KKMS10}: 
\[
\lambda_n(1) < 
\begin{cases}
1.16 , & n=1 , \\
2.76 , & n=2 , \\
4.32 , & n=3 .
\end{cases}
\]
Their 12 digit estimates have been rounded up to 2 decimal places. The numerical estimates obviously satisfy $\lambda_n(1) < n\pi/L$ for $n=1,2,3$, with $L=2$. 

A similarly explicit approach when $\alpha \neq 1$ proceeds through an asymptotic estimate of Kwa\'{s}nicki \cite[Theorem~1]{Kwa12}, which asserts that for the interval of length $L=2$, 
\[
\lambda_n(\alpha) = \Big( \frac{n\pi}{2} - \frac{(2-\alpha)\pi}{8} \Big)^{\! \alpha} + O \Big( \frac{1}{n} \Big) .
\]
Rearranging, we find
\[
\lambda_n(\alpha) = \Big( \frac{n\pi}{2} \Big)^{\! \alpha}
\Big( 1 - \frac{\alpha(2-\alpha)}{4n} + o(1/n) \Big) .
\]
Clearly the second factor on the right is less than $1$ for all large $n$, and so $\lambda_n(\alpha) < (n\pi/L)^\alpha$ for all large $n$. Thus once again we see P\'{o}lya's conjecture fails for the fractional Laplacian. 
\end{proof}

\subsection*{\bf Relation to Laptev's inequality of Berezin--Li--Yau type}
Laptev \cite[Corollary~2.3]{Lap97a} extended Berezin's eigenvalue sum inequality from the Laplacian to the fractional Laplacian, working on general domains and with an even more general class of operators. The resulting lower bound of ``Li--Yau'' form (see \cite[formula (4.2)]{Fra}) says for an interval in $1$-dimension that
\[
\Big( \frac{\pi}{L} \Big)^{\! \alpha} \frac{n^{1+\alpha}}{1+\alpha} \leq \sum_{k=1}^n \lambda_k(\alpha) , \qquad n=1,2,3,\ldots .
\]
For more information, see Frank's survey \cite[Theorem~4.1]{Fra}, and the improvements by Yildirim--Yolcu and Yolcu \cite[Theorem~1.4]{YYY13}, who strengthened the inequality with a lower order term.

Combining this lower bound by Laptev with the upper bound on individual eigenvalues from \autoref{th:polyafalse} yields a two-sided bound, which in the special case $\alpha=1$ has a particularly simple form:
\[
\frac{\pi}{2L} n^2 \leq \sum_{k=1}^n \lambda_k(1) < \frac{\pi}{2L} n(n+1) , \qquad n=1,2,3,\ldots .
\]

\subsection*{\bf Fractional P\'{o}lya conjecture fails for the unit disk, for the first eigenvalue}
Take $n=1$ and consider the unit disk in dimension $d=2$, which has area $\pi$. Then the corresponding term in the Weyl asymptotic \autoref{eq:Weyl} is $(1 \cdot C_2/\pi)^{\alpha/2} = 2^\alpha$. The next theorem shows that the fractional P\'{o}lya conjecture fails already for the first eigenvalue of the disk, when $\alpha$ is not too large. 
\begin{theorem}[Disk]
For the unit disk, $\lambda_1(\alpha) < 2^\alpha$ for all $\alpha \in (0,0.802)$.
\end{theorem}
The theorem can be extended to $\alpha \in (0,0.984)$ provided one accepts  a numerical plot as part of the proof; see part (iii) below.  
\begin{proof}
We rely on several bounds from the literature for the unit ball in $\R^d$. 

(i) The first bound is the simplest, but handles only $\alpha \in (0, 0.699)$. By work of Ba{\~n}uelos and Kulczycki \cite[Corollary~2.2]{BK04},
\[
 \lambda_1(\alpha) \leq \frac{2^{\alpha + 1} \Gamma(\tfrac{\alpha}{2} + 1)^2 \Gamma(\tfrac{d}{2} + \alpha + 1)}{(d + \alpha) \Gamma(\alpha + 1) \Gamma(\tfrac{d}{2})} = \frac{2^{\alpha + 1} (\alpha + 1) \Gamma(\tfrac{\alpha}{2} + 1)^2}{\alpha + 2} 
\]
after substituting the dimension $d=2$. Plotting this bound shows that $\lambda_1(\alpha) < 2^\alpha$ when $\alpha \in (0, 0.699)$. We will not justify this claim rigorously, since part (ii) below gives an analytic proof for an even larger interval of $\alpha$-values. 

(ii) A somewhat stronger estimate by Dyda, Kuznetsov and Kwa{\' s}nicki, namely \cite[formula~(13)]{DKK15a}, says for $d=2$ that
\begin{equation} \label{eq:partii}
 \lambda_1(\alpha) 
\leq \frac{2^{\alpha-1} (\alpha + 2) (7 \alpha + 24) \Gamma(\tfrac{\alpha}{2} + 1)^2}{(\alpha + 4) (\alpha + 6)} .
\end{equation}
By plotting, we verify the desired inequality $\lambda_1(\alpha) < 2^\alpha$ on the larger interval $\alpha \in (0,0.802)$. This inequality can be checked rigorously, as follows: to show  the right side of \autoref{eq:partii} is less than $2^\alpha$ is equivalent to showing
\[
2 \log \Gamma (\tfrac{\alpha}{2} + 2) - \log \frac{\alpha+2}{7\alpha+24} - \log (\alpha+4) - \log (\alpha+6) + \log 2 < 0 .
\]
Each term on the left is convex as a function of $\alpha$, and so it suffices to check that the left side equals $0$ at $\alpha=0$ and is negative at $\alpha=0.802$, which is easily done. 

(iii) To get the desired inequality for the interval $\alpha \in (0,0.984)$, we apply an even stronger (and more complicated) bound of Dyda \cite[Section~5]{Dyd12}. It says for the unit ball that
\[
 \lambda_1(\alpha) \leq \frac{P - \sqrt{P^2 - Q R}}{2 R} 
\]
where the quantities are defined (when $d=2$) by
\begin{align*}
& \hspace*{1.3cm} P 
= \frac{2^{\alpha - 1} \pi^2 (\alpha + 4) (\alpha^2 + 3 \alpha + 6) \Gamma(\tfrac{\alpha}{2} + 1)^2}{(\alpha + 1) (\alpha + 3) (\alpha + 6)} , \\
Q 
= & \frac{4^{\alpha + 1} \pi^2 (\alpha + 2) \Gamma(\tfrac{\alpha}{2} + 1)^4}{\alpha + 6} ,
\qquad 
R = \frac{\pi^2 (\alpha + 4)^2}{4 (\alpha + 1) (\alpha + 2)^2 (\alpha + 3)} ;
\end{align*}
the above formulation is taken from \cite[formula~(12)]{DKK15a}. Substituting these values of $P,Q,R$ and then plotting as a function of $\alpha$ shows $\lambda_1(\alpha) < 2^\alpha$ when $\alpha \in (0,0.984)$. We do not attempt an analytic proof of this last inequality. 
\end{proof}

\subsection*{The square} To disprove P\'{o}lya's conjecture on the square $(-1,1) \times (-1,1)$ of sidelength $2$, it would suffice to show $\lambda_1(\alpha) < (C_2/4)^{\alpha/2} = \pi^{\alpha/2}$. Domain monotonicity of eigenvalues means it would be enough in fact to show the first eigenvalue of the unit disk (which lies inside the square) is less than $\pi^{\alpha/2}$. This last inequality can be verfied when $\alpha < 0.417$ by using the estimate in (iii) above. The simpler bound in (ii) suffices for the square when $\alpha < 0.298$, while the bound in (i) is not good enough for any $\alpha$, for this purpose.

Hence in $2$-dimensions, the fractional P\'{o}lya conjecture can fail even for a tiling domain, namely, the square.

\subsection*{\bf Concluding discussion} We have shown that the analogue of P\'{o}lya's conjecture fails for the fractional Laplacian. The conjecture is known to fail for another variant of the Laplacian too, the so-called magnetic Laplacian, by work of Frank, Loss and Weidl \cite{FLW09}. 

Thus any technique that might prove the original P\'{o}lya conjecture for the Dirichlet Laplacian must be rather special, because it must break down for both the magnetic Laplacian and the fractional Laplacian.

\section*{Appendix. Spectral comparison} \autoref{th:polyafalse} depended on the fact that the eigenvalues of the fractional Laplacian are bounded above by powers of the classical Laplacian eigenvalues. We give a direct proof of this fact in the next Proposition, and then discuss earlier work. The proof relies on Jensen's inequality and the Poincar\'{e} minimax characterization of eigenvalues, and it is new to the best of our knowledge. 
\begin{proposition} \label{pr:spectralcomparison}
The function $\alpha \mapsto \lambda_n(\alpha)^{1/\alpha}$ is strictly increasing when $\alpha>0$, for each $n \geq 1$. Hence $\lambda_n(\alpha) < \lambda_n(2)^{\alpha/2}$ when $0<\alpha<2$.
\end{proposition}
\begin{proof}
Suppose $0<\alpha<\beta<\infty$. Take $u \in H^{\beta/2}(\R^d)$ with $\int_{\R^d} |u|^2 \, dx = 1$, so that $\int_{\R^d} |\widehat{u}(\xi)|^2 \, d\xi = 1$ by Plancherel's identity. Then 
\[
\Big( \int_{\R^d} |\xi|^\alpha |\widehat{u}(\xi)|^2 \, d\xi \Big)^{\! \beta/\alpha} < \int_{\R^d} |\xi|^\beta |\widehat{u}(\xi)|^2 \, d\xi
\]
by Jensen's inequality applied with the strictly convex function $t \mapsto t^{\beta/\alpha}$ and with measure $d\mu(\xi) = |\widehat{u}(\xi)|^2 \, d\xi$, and where the inequality is shown to be strict by the following argument. If equality held then the equality conditions for Jensen would imply that $|\xi|^\alpha$ is constant $\mu$-a.e., meaning $\mu(|\xi| \neq c)=0$ for some constant $c$. Also the sphere $|\xi| = c$ has $\mu$-measure zero, and so we conclude $\mu \equiv 0$ and hence $\widehat{u} = 0$ a.e.\ with respect to Lebesgue measure. That contradiction shows that Jensen's inequality must hold strictly. 

Next, recall that the eigenvalues are characterized variationally \cite[p.~97]{B80}, with
\[
\lambda_n(\alpha) = \min_{S \in S_n(\alpha)} \max \Big\{ \int_{\R^d} |\xi|^\alpha |\widehat{u}|^2 \, d\xi : u \in S \text{\ with\ } \int_{\R^d} |u|^2 \, dx = 1 \Big\} 
\]
for $\alpha > 0$, where $S_n(\alpha)$ is the collection of all $n$-dimensional subspaces of $H^{\alpha/2}_0(\Omega)$. The minimum is attained when $S$ is spanned by the first $n$ eigenfunctions of $(-\Delta)^{\alpha/2}$. 

Choose $S \in S_n(\beta)$ to be the subspace of $H^{\beta/2}_0(\Omega)$ spanned by the first $n$ eigenfunctions of $(-\Delta)^{\beta/2}$. Then $S \in S_n(\alpha)$, just because $H^{\beta/2}_0(\Omega) \subset H^{\alpha/2}_0(\Omega)$, and so the variational characterization and strict Jensen inequality   imply that
\begin{align*}
\lambda_n(\alpha) 
& \leq \max \Big\{ \int_{\R^d} |\xi|^\alpha |\widehat{u}|^2 \, d\xi : u \in S \text{\ with\ } \int_{\R^d} |u|^2 \, dx = 1 \Big\} \\
& < \max \Big\{ \Big( \int_{\R^d} |\xi|^\beta |\widehat{u}|^2 \, d\xi  \Big)^{\! \alpha/\beta} : u \in S \text{\ with\ } \int_{\R^d} |u|^2 \, dx = 1 \Big\} \\
& = \lambda_n(\beta)^{\alpha/\beta} ,
\end{align*}
which completes the proof. 
\end{proof}

Earlier work proved the non-strict inequality $\lambda_n(\alpha) \leq \lambda_n(2)^{\alpha/2}$ for $\alpha=1$ \cite[Theorem~3.14]{BK04}, and for rational $\alpha \in (0,2)$ \cite[Theorem~1.3]{DeB04}, and for general $\alpha \in (0,2)$ \cite[Theorem~3.4]{CS05}. Further, $\alpha \mapsto \lambda_n(\alpha)^{1/\alpha}$ is continuous \cite[Theorem~1.3]{DMH07}, \cite[Example~5.1]{CS06}, and is increasing by work of Chen and Song \cite[Example~5.4]{CS05}, while \autoref{pr:spectralcomparison} shows it is strictly increasing.

A stronger result than \autoref{pr:spectralcomparison} is true when $0<\alpha<\beta=2$: the fractional Laplacian is bounded above as an operator by the $\alpha/2$-th power of the Dirichlet Laplacian. References for the non-strict version of this operator inequality are in Frank's survey paper \cite[Theorem~2.3]{Fra}. For the strict operator inequality, see the paper of Musina and Nazarov \cite[Corollary~4]{MN14}. 

Finally, \autoref{pr:spectralcomparison} and its proof by Jensen's inequality extend to eigenvalues of other families of operators, provided the corresponding Fourier multipliers are related by (strictly) convex transformations, just as $|\xi|^\alpha$ is related to $|\xi|^\beta$ by the transformation $t \mapsto t^{\beta/\alpha}$. Additionally, the result extends from eigenvalues to the more general ``inf--max'' values defined by a variational formula in the case of non-discrete spectrum, although the inequality is no longer strict in that case.

\section*{Acknowledgments}
This research was supported by grants from the Simons Foundation (\#204296 and \#429422 to Richard Laugesen) and the statutory fund of the Department of Mathematics, Faculty of Pure and Applied Mathematics, Wroc{\l}aw University of Science and Technology (Mateusz Kwa\'snicki).

The paper was initiated at the Stefan Banach Mathematical International Center (B\k{e}dlewo, Poland), during the 3rd Conference on Nonlocal Operators and Partial Differential Equations, June 2016. The authors are grateful for the financial support and hospitality received during the conference.

\end{document}